\documentclass{amsart}

\usepackage{graphicx}
\usepackage{dsfont}
\usepackage{textcomp}
\usepackage{amsmath,amssymb,mathtools}
\usepackage{courier}
\usepackage[utf8]{inputenc}
\usepackage{graphicx}
\graphicspath{{images/}}
\usepackage{amsthm}
\usepackage{eucal}
\usepackage{amssymb}
\usepackage{mathrsfs}
\usepackage{dsfont}
\usepackage[usenames,dvipsnames,svgnames,table]{xcolor}
\usepackage{caption}
\usepackage{tikz}
\usepackage{tikz-cd}
\usepackage{wrapfig,lipsum,booktabs}
\usepackage{caption}
\usepackage[numbers]{natbib}

\usepackage{amstext} 
\usepackage{array}   
\usepackage[unicode]{hyperref}
\usepackage{verbatim}
\usepackage{makecell}
\usepackage{float}

\newcolumntype{C}{>{$}c<{$}} 

\newcommand{\N}{\mathbb{N}}
\newcommand{\Z}{\mathbb{Z}}

\renewcommand{\P}{\mathcal{P}}
\newcommand{\NP}{\mathcal{NP}}
\newcommand{\NPC}{\mathcal{NPC}}

\DeclareMathOperator{\End}{End}

\DeclareMathOperator{\Fun}{Fun}
\DeclareMathOperator{\im}{im}

\begingroup
\makeatletter
\@for\theoremstyle:=definition,remark,plain\do{%
	\expandafter\g@addto@macro\csname th@\theoremstyle\endcsname{%
		\addtolength{\thm@preskip}{5pt}
		\setlength{\thm@postskip}{10pt}
	}%
}
\endgroup 

\title{The complexity of intersecting subproducts with subgroups in Cartesian powers}
\author{Pim Spelier}
\date{\today}

\theoremstyle{plain}
\newtheorem{thm}{Theorem}[section] 
\newtheorem{lem}[thm]{Lemma} 
\newtheorem{prop}[thm]{Proposition} 

\theoremstyle{definition}
\newtheorem{defn}[thm]{Definition} 

\theoremstyle{remark}
\newtheorem{rem}[thm]{Remark} 

\usepackage{chngcntr}
\counterwithin{table}{section}

\begin{document}
\begin{abstract}
Given a finite abelian group $G$ and $t\in \N$, there are two natural types of subsets of the Cartesian power $G^t$; namely, Cartesian powers $S^t$ where $S$ is a subset of $G$, and (cosets of) subgroups $H$ of $G^t$. A basic question is whether two such sets intersect. In this paper, we show that this decision problem is NP-complete. Furthermore, for fixed $G$ and $S$ we give a complete classification: we determine conditions for when the problem is NP-complete, and show that in all other cases the problem is solvable in polynomial time. These theorems play a key role in the classification of algebraic decision problems in finitely generated rings developed in \citep{artorders}.
\end{abstract}
\maketitle

\renewcommand{\thefootnote}{\fnsymbol{footnote}} 
\renewcommand{\thefootnote}{\arabic{footnote}}


\section{Introduction}
\label{section:intro}
In this paper, we present a full classification on a natural problem in abelian group theory. Given a finite abelian group $G$ and $t \in \N$, the Cartesian power $G^t$ has two natural types of subset, namely Cartesian products $S^t$ of subsets $S$ of $G$, and cosets $x + H$ of subgroups $H \subset G^t$. While intersecting cosets $x+H$ with other cosets $x + H'$ can be performed in polynomial time, we show that determining whether a coset $x + H$ intersects with a Cartesian products $S^t$ is very often NP-complete, even if $G, S$ are fixed. 

We formally introduce the problems. We phrase the problem in terms of $R$-modules for any commutative ring $R$. For $R = \Z$ we recover the problem on abelian groups mentioned above. For general $R$, the full classification is still open.

\begin{defn}
\label{defn:pigs}
Let $R$ be a commutative ring, let $G$ be a finite $R$-module and $S$ a subset of $G$. Then define the problem $\Pi_{G,S}^R$ as follows: on input $(t,H)$ with $t \in \Z_{\geq 0}$ and $H$ a submodule of $G^t$ given by a list of generators, decide if $H \cap S^t$ is non-empty. Denote $\Pi_{G,S}^\Z$ by $\Pi_{G,S}$.
\end{defn}

\begin{defn}
\label{defn:pgs}
Let $R$ be a commutative ring, let $G$ be a finite $R$-module and $S$ a subset of $G$. Then define the problem $P_{G,S}^R$ as follows: on input $(t,x_*,H)$ with $t \in \Z_{\geq 0}, x_* \in G^t$, and $H$ a submodule of $G^t$ given by a list of generators, decide if $(x_* + H) \cap S^t$ is non-empty. Denote $P_{G,S}^\Z$ by $P_{G,S}$.
\end{defn}

Note that we can without loss of generality assume $R$ is of finite rank over $\Z$, as we can replace $R$ by its image in $\End_\Z(G)$. We remark that $R,G,S$ are \textit{not} part of the input of the problem. In particular, computations inside $G$ can be done in $O(1)$.

These problems certainly lie in $\NP$, as one can give an $R$-linear combination of the generators, and check that it lies in $S^t$. For $R = \Z$, we prove two theorems that completely classify the problems $P_{G,S}$ and $\Pi_{G,S}$, in the sense that for each problem we provide either a polynomial time algorithm or give a proof of NP-completeness.


\begin{thm}[Proposition~\ref{prop:pgseasy},Theorem~\ref{thm:pgsnpc}]
\label{thm:pgs}
If $S$ is empty or a coset of some subgroup of $G$, then we have $P_{G,S} \in \P$. In all other cases, $P_{G,S}$ is NP-complete.
\end{thm}

For the problem $\Pi_{G,S}$, the condition on $G,S$ is slightly different. For example, if $0 \in S$, the intersection will always contain $0$ and the problem is trivial.

\begin{defn}
\label{def:core}
Let $R$ be a ring, let $G$ be a finite $R$-module, and let $S \subset G$ be a subset. We define the \emph{core} $\theta(S)$ of $S$ to be the set
\[
	\theta(S) \coloneq \bigcap_{r \in R \mid rS \subset S} rS.
\]
\end{defn}

\begin{thm}[Theorem~\ref{thm:pigstext}]
\label{thm:pigs}
If $S$ is empty or the core $\theta(S)$ is a coset of some subgroup of $G$, then we have $\Pi_{G,S} \in \P$. In all other cases, $\Pi_{G,S}$ is NP-complete.
\end{thm}

\begin{rem}
\label{rem:psi}
Note that if $0 \in S$, then $\theta(S) = \{0\}$. Additionally, if $G$ is a group with order a prime power and $S$ does not contain 0, then $\theta(S) = S$, as will be proven in Lemma~\ref{lem:psipowerp}.
\end{rem}

In the follow-up paper \citep{artorders} we use these results to determine the complexity of finding roots of a fixed polynomial $f \in \Z[x]$ in finitely generated rings. 

\subsection{Acknowledgements}
This project grew out of the thesis of the author \cite{spelier2018}. It is a pleasure to thank my thesis supervisors Hendrik Lenstra and Walter Kosters for their help. I am also very thankful to Daan van Gent for their comments on a preliminary version of the paper.

\section{Reductions and polynomiality}
\label{section:group}
In this section we will prove some preliminary lemmas on $P_{G,S}$ and $\Pi_{G,S}$. Some of the lemmas we use to prove Theorems \ref{thm:pgs} and \ref{thm:pigs} we give for general $P_{G,S}^R$ (resp. $\Pi_{G,S}^R$) and some only for $P_{G,S}$ (resp. $\Pi_{G,S}$). Throughout this section, let $R$ be a commutative ring, finitely generated as a $\Z$-module. All $R$-modules we consider in this section are finite.

We use the following notation for reductions between computational problems.
\begin{defn}
Let $P,Q$ be two problems. We write $P \leq Q$ if there is a polynomial-time reduction from $P$ to $Q$. We write $P \approx Q$ if $P \leq Q$ and $Q \leq P$.
\end{defn}


\begin{lem}
\label{lem:transinvar}
We have $P_{G,S}^R \approx P_{G,S+g}^R$ for all $g \in G$.
\end{lem}
\begin{proof}
For the reduction $P_{G,S}^R \leq P_{G,S+g}^R$, we send an instance $(t,x_*,H)$ to $(t,x_* + (g,\ldots,g),H)$. By symmetry, we also have $P_{G,S+g}^R \leq P_{G,S}^R$; by the definition of $\approx$, we are done.
\end{proof}

\begin{lem}
\label{lem:restrict}
If $G'$ is a submodule of $G$, then we have $P_{G',G' \cap S}^R \leq P_{G,S}^R$.
\end{lem}
\begin{proof}
Given an instance $(t,x_*,H)$ of the first $P_{G',G' \cap S}^R$, we see it is also an instance of $P_{G,S}^R$, and as $H \cap S^t \subset G'^t \cap S^t = (G' \cap S)^t$, we see it is a yes-instance of the first problem exactly if it is a yes-instance of the second one.
\end{proof}

\begin{lem}
\label{lem:divideout}
Let $G'$ be a submodule of $G$ and $S'$ a subset of $G$, and define $S = S' + G'$. Then we have $P_{G/G',S'}^R \approx P_{G,S}^R$.
\end{lem}
\begin{proof}
For the reduction $P_{G/G',S'}^R \leq P_{G,S}^R$ we send an instance $(t,x_*,H)$ to $(t,x_*,H + G'^t)$; this works exactly because of the property $S = S' + G'$. For the reduction $P_{G,S}^R \leq P_{G/G',S'}^R$, we pass everything through the map $G \rightarrow G/G'$.
\end{proof}

For the last lemma, we first introduce a definition.

\begin{defn}
Let $G$ be an $R$-module. A \emph{transformation} on $G$ is a map $\varphi : G \to G$ of the form $x \mapsto c(x) + g$ where $c$ is an $R$-linear endomorphism of $G$ and $g \in G$. For $S \subset G$, we write $S_{\varphi}$ for $S \cap \varphi^{-1}(S)$.
\end{defn}

\begin{lem}
\label{lem:transformation}
Let $G$ be an $R$-module, $S$ a subset of $G$ and $\varphi$ a transformation on $G$. Then $P_{G,S_{\varphi}}^R \leq P_{G,S}^R$.
\end{lem}
\begin{proof}
Let $c$ be an $R$-linear endomorphism of $G$, and $g \in G$ such that $\varphi$ is given by $x\mapsto c(x) + g$. Let $(t,x_*,H)$ be an instance of $P_{G,S_{\varphi}}$. Define $\Gamma = G^{t} \times G^t$ with $\pi_1,\pi_2$ the two projections, let $x_*' = (x_*,c(x_*) + g) \in \Gamma$ and $H' = \{(h,c(h)) \mid h \in H\} \subset \Gamma$. Note that $H$ is naturally isomorphic to $H'$ by $f: h \mapsto (h,c(h))$, as $R$ is commutative. We then see that for $h \in H$ we have that $x_*' + f(h) \in S^{2t}$ if and only if $\pi_1(x_*' + f(h)),\pi_2(x_*' + f(h)) \in S^{t}$ if and only if $x_* + h \in S$ and $c(x_*+h)+g = \varphi(x_* +h) \in S$, which is equivalent to $x_* + h \in S_{\varphi}$. This shows that $(2t,x_*',H')$ is a yes-instance of $P_{G,S}^R$ if and only if $(t,x_*,H)$ is a yes-instance of $P_{G,S_{\varphi}}^R$.
\end{proof}

We will now prove the polynomiality result of Theorem~\ref{thm:pgs}.
\begin{prop}
\label{prop:pgseasy}
If $S \subset G$ is empty or a coset of some subgroup of $G$, then $P_{G,S}^R \in \P$.
\end{prop}
\begin{proof}
As a submodule is in particular a subgroup, we have the inequality $P_{G,S}^R \leq P_{G,S}$, so it suffices to prove the lemma assuming that $R = \Z$. If $S$ is empty, then the problem is easy --- the intersection is always empty for $t > 0$ and non-empty for $t = 0$. If $S$ is a coset of some subgroup $G'$, by Lemma~\ref{lem:transinvar} we may assume $S = G'$, and by Lemma~\ref{lem:divideout} the problem is equivalent to $P_{G/G',\{0\}}$. To solve $P_{G/G',\{0\}}$ in polynomial time, we only need to decide whether the single element $-x_*$ is in $H$: this is simply checking whether a linear system of equations over $\Z$ has a solution, which can be done in polynomial time as proven in \citep[\textsection14]{lenstra2008lattices}. Hence we indeed find that $P_{G,S}$ admits a polynomial time algorithm.
\end{proof}

\section{NP-completeness}

We will prove the NP-complete part of Theorem~\ref{thm:pgs} by induction on $|S|$. There are two base cases: $|S| = 2$ for any group $G$, and $|S| = |G| - 1$ for $G = C_2^2$ where $C_2$ is the cyclic group of order $2$. Both cases follow from a reduction from $n$-colorability.
\begin{defn}
Let $n \in \Z_{\geq 1}$ be given. We define the problem $n$-colorability as follows. Given as input a graph $(V,E)$, decide whether there exists a mapping $V \to \{1,\dots,n\}$ such that adjacent vertices have different images.
\end{defn}

\begin{prop}
\label{prop:s3}
Let $G$ be an $R$-module of cardinality at least $3$, and $S$ a subset of cardinality $|G|-1$. Then $P_{G,S}^R$ is NP-complete.
\end{prop}
\begin{proof}
By translating, we can assume $S = G \setminus \{0\}$. We will reduce from $|G|$-colorability.

Let $(V,E)$ be an instance of $|G|$-colorability. Note that $G^V = \{(g_v)_{v\in V} \mid g_v \in G\}$ can be thought of as all ways of assigning elements of $G$ to the vertices. Let $f$ be the homomorphism from $G^V$ to $G^E$ defined by $(g_v)_{v\in V} \mapsto (g_u-g_v)_{(u,v) \in E}$, and note that an assignment in $G^V$ is a $|G|$-coloring if and only if it is sent to an element of $S^E$. Then we define the instance $(t,x_*,H)$ of can take $H$ to be the submodule of $G^E$ generated by the images of $R$-generators of $G^V$, of which we need at most $|G||V|$. This is a valid reduction as $(V,E)$ will be $|G|$-colorable if and only if $H \cap S^E \not=\varnothing$; we can take $x_*$ to be zero. 

As we have $|G| \geq 3$, the $|G|$-colorability problem is NP-complete \citep{garey} hence $P_{G,S}$ is NP-complete.
\end{proof}

\begin{prop}
\label{prop:s2}
Let $G$ be an $R$-module and $S$ a subset of cardinality $2$ which is not a coset of a subgroup. Then $P_{G,S}^R$ is NP-complete.
\end{prop}
\begin{proof}
Write $S = \{s,s+d\}$. Because of Lemma~\ref{lem:transinvar} we can take $s = 0$. Since $S$ is of cardinality $2$ and not a subgroup, we have $-d,2d \not\in S$. By Lemma~\ref{lem:restrict} we are allowed to take $G = Rd$, and by renaming we can take $d = 1$, the module $G$ some finite quotient of $R$ and $S = \{0,1\}$ with $-1,2 \not\in S$.

We reduce from $3$-colorability. Let $C$ be our set of three colors. Given a graph $(V,E)$, we will construct a subgroup $H \subset \Gamma := G^{V \times C} \times G^{V} \times G^V \times G^{E\times C}$ and $x_* \in \Gamma$ such that $H + x_*$ has an element in $T := S^{V \times C} \times S^{V} \times S^V \times S^{E\times C} $ exactly if $(V,E)$ is 3-colorable. Let $\pi_1,\pi_2,\pi_3,\pi_4$ denote the four projections from $\Gamma$ on the four factors $G^{V \times C},G^{V}  ,G^V ,G^{E\times C}$. We take $H$ to be the image of $G^{V \times C}$ under the map 
\begin{align*}
    \varphi: G^{V \times C} &\to \Gamma \\
    f &\mapsto (f,\sigma(f),\sigma(f),\tau(f)) \\
\end{align*}
where we define
\begin{align*}
    \sigma(f)(v) &= \sum_{c \in C} f(v,c)\\
    \tau(f)(e,c) &= \sum_{v \in e} f(v,c).
\end{align*} Furthermore, we choose $x_* = (0,0,-1,0)$. We see $\varphi$ gives an isomorphism $G^{V \times C} \to H$, with inverse $\pi_1$. Note that $\pi_1(\varphi(f) + x_*)$ needs to be in $\{0,1\}^{V \times C}$ for $\varphi(f) + x_*$ to be in $T$, and $\pi_1(\varphi(f) + x_*)\in \{0,1\}^{V \times C}$ happens if and only if $f$ itself is in $\{0,1\}^{V \times C}$.

To prove this is truly a reduction, we interpret $\{0,1\}^{V \times C}$ as assignments of subsets of $C$ to the vertices $V$, using the bijection between $\Fun(V,\{0,1\}^C)$ and $\Fun(V\times C,\{0,1\}) = \{0,1\}^{V \times C}$. A $3$-coloring of $(V,E)$ can then be equivalently redefined as such an assignment $f \in \{0,1\}^{V \times C}$ with the property that for every vertex $v \in V$ we have $\sigma(f)(v) =1$, i.e., each vertex gets a single color, and that for every $c \in C,\{i,j\} \in E$ we have $f(i)(c),f(j)(c)$ not both $1$. It suffices to show that these colorings map under $\varphi$ exactly to those $h \in \varphi(\{0,1\}^{V \times C})$ with $h + x_* \in T$.

Let $f \in \{0,1\}^{V \times C}$ be such a coloring with subsets of $C$, and let $h = \varphi(f)$ be the corresponding element of $H$. Then note that $\pi_2(h + x_*) \in \{0,1\}^{V}$ if and only if $\sigma(f)(v)$ is either $0$ or $1$. Also, $\pi_3(h+x_*)(v) = \sigma(f)(v) - 1$, which is $-1 \not\in S$ if $\sigma(f)(v) = 0$. Hence $(\pi_2(h+x_*),\pi_3(h+x_*)) \in \{0,1\}^{V} \times \{0,1\}^{V}$ if and only if for every vertex $v \in V$ we have $\sigma(f)(v) = 1$.

Finally, note that $\pi_4(h + x_*)(e,c)$ is in $\{0,1\}$ exactly if the two endpoints of $e$ do not both have color $c$. This completes the proof that the elements of $\{0,1\}^{V \times C}$ that are $3$-colorings correspond to $h \in H$ with $h + x_* \in T$, and hence the reduction is completed.
\end{proof}

First, we will do the induction step for a special family of finite groups: $G = \langle a,b \rangle$ with $a\not= b$ and $S$ containing $0,a,b$, but not $a+b$. 

\begin{prop}
\label{prop:pigsspecial}
Let $G$ be a finite abelian group generated by two distinct elements $a, b$, and $S$ a subset of $G$ containing $0,a,b$ but not $a + b$. Assume Theorem~\ref{thm:pgs} holds for all $P_{G',S'}$ with $|G'| + |S'| < |G| + |S|$. Then $P_{G,S}$ is NP-complete.
\end{prop}
\begin{proof}
In this proof, we will heavily use Lemma~\ref{lem:transformation}. We restrict to bijective transformations of the form $x \mapsto cx + g$ with $c = \pm 1$. If $\varphi$ is a transformation on $G$ and $P_{G,S_{\varphi}}$ is NP-complete, so is $P_{G,S}$. For the NP-completeness of the former, we only need $2 \leq |S_{\varphi}| < |S|$ and $S_{\varphi}$ not a coset, and then we are done by the induction hypothesis. We can also interpret this in another way: if two of the three conditions on $S_{\varphi}$ hold, then either we are done immediately, or the third one does not hold, which gives us more information about $S$. If $\varphi$ is bijective, then  $|S_{\varphi}| \leq |S|$ with equality if and only if $\varphi$ induces a permutation of $S$.

We will now prove the following claim: let $A = \{g \in G \mid g,g+a,g+b \in S, g+a+b \not\in S\}$. We already know that $0 \in A$. We will prove that if $g \in A$, then $\Pi_{G,S}$ is NP-complete or $g - 2a \in A$. For the proof, we can by Lemma~\ref{lem:transinvar} assume that $g = 0$.

We do a case distinction, based on whether $a-b$ is in $S$ or not. First, we assume it is. Let $\varphi_1: x\mapsto a+b-x$, and note that $a,b\in S_{\varphi_1},0\not\in S_{\varphi_1}$, so either we are done or $S_\phi \ni a,b$ is a coset, implying that $a + \langle b-a \rangle \subset S$, which we now assume. Now let $\varphi_2$ be the transformation  $x \mapsto a-b + x$. Note $0,a,b\in S_{\varphi_2}$, so $2 \leq |S|$ and $S_{\varphi_2}$ is not a coset. This now tells us that either we are done, or $S = S_{\varphi_2}$, meaning we can write $S = \Sigma + \langle b-a \rangle$ for some set $\Sigma$. Writing $\Gamma = G/\langle b-a \rangle$ we see $\Sigma$ is not a coset in $\Gamma$. Furthermore by Lemma~\ref{lem:divideout} we know $P_{G,S} \approx P_{\Gamma,\Sigma}$. Since $b-a \not= 0$, we have that $|\Sigma| + |\Gamma|$ is strictly smaller than $|S| + |G|$, which means that by the induction hypothesis we know $P_{\Gamma,\Sigma}$ to be NP-complete. Hence $P_{G,S} \in \NPC$ as we wanted to show.

In the remaining case, we have $a-b\not\in S$ and similarly we can assume that $b-a \not\in S$ holds as well. Looking at $x \mapsto b-x$ or $x \mapsto a -x$ we see we can assume $\langle a \rangle, \langle b \rangle \subset S$. Now we look at $\varphi_3: x\mapsto x-a-b$. If $-a-b\in S$, all conditions are met and we are done. So assume $-a-b\not\in S$. Finally taking $\varphi_4: x\mapsto a-b+x$, we can see that we must have $-a + \langle a-b \rangle \subset S$. We now have have $0-2a,a-2a,b-2a\in S, a+b-2a\not\in S$, hence $-2a \in A$.  This proves the claim.

Now $A$ is closed under $g \mapsto g - 2a$ and by symmetry also under $g \mapsto g- 2b$. As $a,b$ are of finite order, we find $2G \subset A$ and hence $S = \{ka + \ell b \mid k,\ell\in\Z, k\ell \equiv 0 \bmod 2\}$ and $G \setminus S = a+b+\langle 2a, 2b\rangle$, meaning $A = 2G$. Dividing out by $A$ and using Lemma~\ref{lem:divideout} we see we $P_{G,S}$ is equivalent to $P_{G/A,S'}$ where $S' = \{0,a,b\}$. Note that $|G/A| = 4$; we know $0,a+b$ are different in $G/A$ as $x \in \langle 2a, 2b\rangle$ implies $x \in S$, and then $a$ is non-zero as we have that $0+b \in S$ but $a + b \not\in S$, hence $a \not \in \langle 2a, 2b\rangle$. So $G/A = C_2^2$ and $|S'| = |G/A|-1$. We have already proven this to be NP-complete in Proposition~\ref{prop:s3}, so we are done.
\end{proof}

Finally, we will prove Theorem~\ref{thm:pgs} in the general case, by reducing to the case in Proposition~\ref{prop:pigsspecial}. For this, we first prove the following little lemma.
\begin{lem}
\label{lem:sab}
Let $G$ be an abelian group, and $S$ a subset of $G$. If $S$ has at least three elements, and the following statement holds
\[
\forall s,a,b: \left(s,s+a,s+b \in S \wedge a\not=b\right) \Rightarrow s+a+b\in S,
\]
then $S$ is a coset.
\end{lem}
\begin{proof}
Since the statement is translation invariant, assume $0 \in S$; we will prove that $S$ is a subgroup. Let $\{0,x,y\}$ be a subset of $S$ of size three, i.e. $x,y$ non-zero and different. As per the property for $s =0, a = x, b = y$, we already have $x+y \in S$, it suffices to prove that $x+x,-x \in S$. As $x+y\in S$, we can apply the property with $(x+y,-x,-x-y)$ to see $-x \in S$ and with $(y,x,-y+x)$ to get $2x \in S$, concluding the proof.
\end{proof}

\begin{thm}
\label{thm:pgsnpc}
Let $G$ be a finite abelian group, and $S \subset G$ a non-empty subset that is not a coset. Then $P_{G,S}$ is NP-complete.
\end{thm}
\begin{proof}
If $|S| = 2$, this is Proposition~\ref{prop:s2}. If $|S| \geq 3$, we can by contraposition of Lemma~\ref{lem:sab} find $s,a,b$ with $s,s+a,s+b\in S$ and $a\not=b$ and $s+a+b\not\in S$; by translating, we can assume $s = 0$. Then, we set $G' = \langle a,b \rangle$ and $S' = S \cap G'$. By Proposition~\ref{prop:pigsspecial}, we know $P_{G',S'}$ is NP-complete, and then by Lemma~\ref{lem:restrict} we find $P_{G,S}$ is NP-complete, as we wanted to show.
\end{proof}

\subsection{Proof of Theorem~\ref{thm:pigs}}
\label{subs:pigs}
We will show that Theorem~\ref{thm:pigs} follows from Theorem~\ref{thm:pgs} and an equivalence of problems. This equivalence holds in fully generality of $R$-modules. Recall the \emph{core} $\theta(S)$ of $S$ is the subset
\[
\theta(S) = \bigcap_{r \in R \mid rS \subset S} rS
\]
of $G$, as per Definition~\ref{def:core}.

\begin{lem}
\label{lem:pgspigs}
With $G$ a finite $R$-module, $S \subset G$ we have the following equivalence of problems
\[
\Pi_{G,S}^R = \Pi_{G,\theta(S)}^R \approx \Pi_{R\theta(S),\theta(S)}^R \approx P_{R\theta(S),\theta(S)}^R
\]
where $R\theta(S)$ means the $R$-module generated by $\theta(S)$.
\end{lem}
\begin{proof}
We have to prove three equivalences, where the first is an equality. Recall that two problems are equal if they have the same set of instances and the same set of yes-instances.

For the first one, note that if $(t,H)$ is a yes-instance of $\Pi_{G,S}^R$ with certificate $h \in H \cap S^{t}$, then 
\[
\left(\prod_{r \in \im(R \to \End(G)) : rS \subset S} r\right)  h
\]
is in $\theta(S)^t$ as $R$ is commutative, hence $(t,H)$ is a yes-instance of $\Pi_{G,\theta(S)}^R$. The other way around, if $(t,H)$ is a yes-instance of $\Pi_{G,\theta(S)}^R$, then it is a yes-instance of $\Pi_{G,S}^R$ since $\theta(S)$ is a subset of $S$, proving the first equality.

For the second one, write $S' = \theta(S), G' = RS'$ and note that $\Pi_{G,S'} \leq \Pi_{G',S'}$ by taking any instance $H$ of the first problem and intersecting it with $G'^t$ using the kernel algorithm from \citep[\textsection14]{lenstra2008lattices}, since $H \cap S'^t = (H \cap G'^t) \cap S'^t$. And by Lemma~\ref{lem:restrict}, the inequality $\Pi_{G',S'} \leq \Pi_{G,S'}$ holds as well.

The third equivalence requires a more complicated reduction. Note $\Pi_{G',S'}^R \leq P_{G',S'}^R$ by taking $x_* = 0$. To show $P_{G',S'}^R \leq \Pi_{G',S'}^R$, let $(t,x_*,H)$ be an instance of $P_{G',S'}^R$. Let $n = |S'|$, and write $S' = \{s_1,\dots,s_n\}$. We construct an instance $(t',H')$ of $\Pi_{G',S'}$ with $t' = t + n$ and $H' = H \times \{0\}^t + R\cdot y_*$ where $y_*$ is defined as $(x_*,s_1,\dots,s_n)$.

We need to check that if $(t,x_*,H)$ is a yes-instance, so is $(t',H')$ and vice versa. For the first implication; if $h \in H$ has $x_* + h \in S^t$, then $h' = y_* + (h,0)$ is an element of $H'$, and lies in $S'$ on every coordinate, hence we see $h' \in H' \cap S'^{t'}$. For the other implication, let $h' = ay_* + (h,0)$ be an element of $H' \cap S'^{t'}$. Looking at the last $|S'|$ coordinates, we see $aS' \subset S'$. But as $\theta(S') = \theta^2(S) = \theta(S) = S'$, we must have $aS' = S'$. Since $a$ induces a bijection on $S'$ and $S'$ generates $G'$ we see $a$ is a unit in $\End(G')$, using the commutativity of $R$. Then some power of $a$ is its inverse in $\End(G')$. Hence we can multiply $h'$ with $a^{-1} \in \End(G')$, and since $a^{-1}S' = S'$ we see $y_* + a^{-1}(h,0) \in S'^{t'}$. Restricting to the first $t$ places, we see $x_* + h \in S'^{t}$, hence $(x_* + H) \cap S'^t \not= \varnothing$ as we wanted to show.
\end{proof}

As a corollary, we obtain the following theorem.

\begin{thm}
\label{thm:pigstext}
If $S$ is empty or the core $\theta(S)$ is a coset of some subgroup of $G$, then we have $\Pi_{G,S} \in \P$. In all other cases, $\Pi_{G,S}$ is NP-complete.
\end{thm}

We will also prove a short lemma about $\theta$ for a special family of $R$-modules.
\begin{lem}
\label{lem:psipowerp}
Let $G$ be an $R$-module such that $A := \im(R \to \End(G))$ is local. Let $S$ be a subset of $G$. Then $\theta(S)$ equals $\{0\}$ if $0 \in S$ and $S$ otherwise.
\end{lem}
\begin{proof}
Obviously, if $0 \in S$ then for every integer $a$ we have $0 \in aS$ so $\{0\} \subset \theta(S)$, and $0S \subset S$ hence $\theta(S) \subset \{0\}$, proving the first part. For the second part, in a local finite ring the powers of the maximal ideal must stabilise, which by Nakayama's lemma mean they must become zero. So every element of $A$ is either invertible or nilpotent. If $r \in A$ is nilpotent, and $0 \not\in S$, then $rS \not\subset S$; otherwise, we would have $r^k S \subset S$ for every $k \in \Z_{> 0}$, contradicting with the nilpotency of $r$ and $0 \not \in S$. That means that if $rS \subset S$ then on $G$, we have that $r$ induces an automorphism, and $rS \subset S$ then implies by cardinality that $rS = S$. Hence in this case $\theta(S) = S$, as we set out to prove.
\end{proof}
\begin{rem}
Some important examples of when the conditions are satisfied, are the case where $R$ itself is local, and the case where $R = \Z$ and $G$ has prime power cardinality.
\end{rem}

\bibliographystyle{plain}
\addcontentsline{toc}{section}{References}
\bibliography{references.bib}


\end{document}